\documentclass[a4paper,article]{amsart}
\usepackage{amsfonts}
\usepackage{amssymb}
\usepackage{amsthm}
\usepackage{enumerate}
\usepackage[english]{babel}
\usepackage{hyperref}
\usepackage[T1]{fontenc}
\usepackage[utf8]{inputenc}
\usepackage{caption}
\usepackage{mathtools}
\usepackage[usenames,dvipsnames]{pstricks}

\usepackage{color}

\theoremstyle{plain}
\begingroup
\theoremstyle{plain}
\newtheorem{theorem}{Theorem}[section]
\newtheorem{corollary}[theorem]{Corollary}
\newtheorem{proposition}[theorem]{Proposition}
\newtheorem{lemma}[theorem]{Lemma}
\theoremstyle{definition}
\newtheorem{definition}[theorem]{Definition}

\theoremstyle{remark}
\newtheorem{remark}[theorem]{Remark}
\theoremstyle{claim}
\newtheorem{claim}[theorem]{Claim}
\theoremstyle{definition}
\newtheorem{notation}[theorem]{Notation}
\endgroup

\theoremstyle{definition}
\theoremstyle{remark}

\mathchardef\emptyset="001F

\numberwithin{equation}{section}

\makeatletter
\def\Ddots{\mathinner{\mkern1mu\raise\p@
\vbox{\kern7\p@\hbox{.}}\mkern2mu
\raise4\p@\hbox{.}\mkern2mu\raise7\p@\hbox{.}\mkern1mu}}
\makeatother

\title[]
{Switching in time-optimal problem. \\ The 3-D case with 2-D control.}
\author[A. A. Agrachev]{Andrei A. Agrachev}
\address[A. A. Agrachev]{SISSA, 34136 Trieste,  Italy; Steklov Mathematical Institute, 119991 Moscow, Russia}
\email[A. Agrachev]{agrachev@sissa.it}
\author[C. Biolo]{Carolina Biolo}
\address[Carolina Biolo]{SISSA, Via Bonomea 265, 34136 Trieste, Italy}
\email[Carolina Biolo]{cbiolo@sissa.it}
\date{}

\begin{document}

\begin{abstract}
We study local structure of time-optimal controls and trajectories for a 3-dimensional control-affine system with a 2-dimensional control parameter with values in the disk. In particular, we give sufficient conditions, in terms of Lie bracket relations, for optimal controls to be smooth or to have only isolated jump discontinuities.
\end{abstract}
\maketitle
\tableofcontents
\section{Introduction}

This paper is a one more step towards the understanding of the structure of time-optimal controls and trajectories for control affine systems of the form:
$$
\dot q=f_0(q)+\sum_{i=1}^ku_if_i(q),\quad q\in M,\ u\in U, \eqno (1)
$$
where $M$ is a smooth $n$-dimensional manifold, $U=\{(u_1,\ldots,u_k):\sum\limits_{i=1}^k u_i^2\le 1\}$ is a $k$-dimensional ball, and $f_0,\,f_1,\,\ldots,\,f_k$ are smooth vector fields. We also assume that $f_1(q),\ldots,f_k(q)$ are linearly independent in the domain under consideration.

The case $k=n$ is the Zermelo navigation problem: optimal controls are smooth in this case. In more general situations, discontinuous controls are unavoidable and, in principle, any measurable function can be an optimal control. Therefore, it is reasonable to focus on generic ensembles of vector fields $f_0,f_1,\ldots,f_k$ and thus avoid a pathologic behavior.

If $k=1,\ n=2$, then, for generic pair of vector fields $f_0,f_1$, any optimal control is piecewise smooth; moreover, any point in $M$ has a neigborhood such that all optimal trajectories contained in the neighborhood are concatenations of at most 2 smooth pieces (have at most one switching in the control-theoretic terminology), see \cite{BP} and \cite{Suss}.
The complexity of optimal controls grows fast with $n$. For $k=1,\ n=3$ generic situation is only partially studied (see \cite{Sch}, \cite{Suss2} and \cite{AS}): we know that any point out of a 1-dimensional Whitney-stratified subset of ``bad points'' has a small neighborhood that contains only optimal trajectories with at most 4 switchings. We still do not know if there is any bound on the number of switchings in the points of the ``bad'' 1-dimensional subset. We know however that the chattering phenomenon (a Pontryagin extremal with a convergent sequences of switching points) is unavoidable for $k=1$ and sufficiently big $n$, see \cite{Ku} and \cite{ZB}.

In this paper, we study the case $k=2,\ n=3$. In particular, for generic triple $(f_0,f_1,f_2)$ we obtain that
any point out of a discrete subset of ``bad points'' in $M$ has a neighborhood such that any optimal trajectory contained in the neighborhood has at most one switching.

Actually we have much more precise results about the structure of optimal controls formulated in Theorems 3.1, 3.5. In particular, we compute the right and the left limits of the control in the switching point in term of the Lie bracket relations. Moreover, we expect that the developed here techniques is efficient  also in the case $k=n-1$ with an arbitrary $n$ and that, in general, complexity of the switchings depends much more on $n-k$ than on $n$.

\section{Preliminaries} \label{sec:preliminaries.section}

In this section we recall some basic definitions in Geometric Control Theory. For a more detailed introduction, see \cite{A}.
\begin{definition}
Given a $n$-dimensional manifold $M$, we call $\mathrm{Vec}(M)$ the \emph{set of the smooth vector fields} on $M$, i.e. each $f\in \mathrm{Vec}(M)$ is a smooth map with respect to $q\in M$
$$f:M\longrightarrow TM
$$
such that if $q\in M$ then $f(q)\in T_q M$.
\end{definition}
\begin{definition}
A \emph{smooth dynamical system} on $M$ is defined by an ordinary differential equation
\begin{equation}
\label{dynamical.system}
\dot{q}=f(q)
\end{equation}
where $q\in M$, $f\in \mathrm{Vec}(M)$.\\
A \emph{solution} of (\ref{dynamical.system}) is a map
$$q:I\longrightarrow M
$$
with $I\subseteq\mathbb{R}$ interval, such that it holds
$$\frac{d}{dt}q(t)=f(q(t))
$$
for every $t\in I$.
\end{definition}
\begin{theorem}
Given a $n$-dimensional manifold $M$ and (\ref{dynamical.system}) a smooth dynamical system on $M$, for each initial point $q_0\in M$ there exists a unique solution $q(t,q_0)$ on $M$, defined in an interval $I\subseteq\mathbb{R}$ small enough, such that $q(0,q_0)=q_0$.
\end{theorem}
\begin{definition}
$f\in \mathrm{Vec}(M)$ is a \emph{complete vector field} if for each $q_0\in M$ the solution $q(t,q_0)$ of (\ref{dynamical.system}) is defined for every $t\in \mathbb{R}$.
\end{definition}
\begin{remark}$f\in \mathrm{Vec}(M)$ with a compact support is a complete vector field.
\end{remark}
\begin{remark}Since we are interested in the local behaviour of trajectories, during all this work we consider directly complete vector fields.
\end{remark}

\begin{definition}
Given a manifold $M$ and $U\subseteq\mathbb{R}^m$ a set, a \emph{control system} is a family of dynamical systems
$$\dot{q}=f_u(q)
$$
where $q\in M$ and $\{f_u\}_{u\in U}\subseteq \mathrm{Vec}(M)$ is a family of vector fields on $M$ parametrized by $u\in U$.\\
U is called \emph{space of control parameters}.
\end{definition}
We are interested in controls, which change during the time.
\begin{definition}
An \emph{admissible control} is a map measurable and essentially bounded
$$\begin{array}{rcc}
u:(t_1,t_2) & \longrightarrow & U\\
t & \longmapsto& u(t),
\end{array}
$$
from a time interval $(t_1,t_2)$ to $U$.\\
We call $\mathcal{U}$ the \emph{set of admissible controls}.
\end{definition}

Therefore, we consider the following control system in $M$
\begin{equation}
\label{control.system}
\dot{q}=f_u(q)
\end{equation}
where $q\in M$ and $\{f_u\}_{u\in \mathcal{U}}\subseteq \mathrm{Vec}(M)$, with $u$ admissible control.\\ \\
With the following theorem we want to show that it is guarantied the locally existence and uniqueness of the solution of the control system which we are considering, for every initial point, choosing an admissible control.
\begin{theorem}
Fixed an admissible control $u\in\mathcal{U}$, (\ref{control.system}) is a non-autonomous ordinary differential equation, where the right-hand side is smooth with respect to $q$, and measurable essentially bounded with respect to $t$, then, for each $q_0\in M$, there exists a local unique solution $q_u(t,q_0)$, depending on $u\in \mathcal{U}$, such that $q_u(0,q_0)=q_0$ and it is lipschitzian with respect to $t$.
\end{theorem}

\begin{definition}
\label{007}
We denote $q_u(t,q_0)$ the \emph{admissible trajectory} solution of (\ref{control.system}), chosen $u\in \mathcal{U}$, and
$$A_{q_0}=\{q_u(t,q_0) : t\geq 0,u\in\mathcal{U} \}
$$
the \emph{attainable set} from $q_0$. \\
Moreover we will write $q_u(t)=q_u(t,q_0)$ if we do not need to stress  that the initial position is $q_0$.
\end{definition}
In this paper, we are going to study an affine control system.
\begin{definition}
An\emph{ affine control system} is a control system with the following form
\begin{equation}
\label{affine.control.system}
\dot{q}=f_0(q)+u_1f_1(q)+\ldots+u_{k}f_{k}(q)
\end{equation}
where $f_0\ldots f_k$ $\in \mathrm{Vec}(M)$ and $u=(u_1, \ldots, u_k)\in \mathcal{U}$ is an admissible control which takes value in the set $U\subseteq \mathbb{R}^k$. The uncontrollable term $f_0$ is called \emph{drift}. \\Moreover, we can consider the $n\times k$ matrix
$$f(q)=\left( \begin{array}{c}f_1(q),\ldots,f_k(q)\end{array}\right)
$$
and rewrite the system (\ref{affine.control.system})
$$\dot{q}=f_0(q)+f(q)u.
$$
\end{definition}

\subsection{Time-optimal problem}
Let us introduce the time-optimal problem.
\begin{definition}
Given the control system (\ref{control.system}), $q_0\in M$ and $q_1\in A_{q_0}$, the \emph{time-optimal problem} consists in minimizing the time of motion from $q_0$ to $q_1$ via admissible trajectories:
\begin{equation}
\label{time-optimal.problem}\left\lbrace \begin{array}{ll}
\dot{q}=f_u(q)&u\in \mathcal{U}\\
q_u(0,q_0)=q_0&\\
q_u(t_1,q_0)=q_1&\\
t_1 \rightarrow \min&
\end{array} \right.
\end{equation}
We call these minimizer trajectories \emph{time-optimal trajectories}, and \emph{time-optimal controls} the correspondent controls.
\end{definition}
\subsubsection{Existence of time-optimal trajectories}
Classical Filippov's Theorem (See \cite{A}) guarantees the existence of a time-optimal control for the affine control system if $U$ is a convex compact and $q_0$ is sufficiently close to $q_1$.
\subsection{First and second order necessary optimality condition}
Now we need to introduce basic notions about Lie brackets, Hamiltonian systems and Poisson brackets, so that we can present the very important first and second order necessary conditions for optimal trajectories: Pontryagin Maximum Principle, and Goh condition.
\begin{definition}
Let $f,g\in \mathrm{Vec}(M)$, we define their \emph{Lie brackets} the following vector field
$$[f,g](q)=\frac{\partial}{\partial t}_{|t=0}e_*^{-t f}g(q), \quad \forall q\in M
$$
where $e_*^{-t f}$ is the push forward of the flow $e^{-t f}$, defined by $f$.\\
Therefore, we present an equivalent definition of Lie brackets, which will help us to understand their geometric meaning:
$$[f,g](q)=\frac{\partial}{\partial t}_{|t=0}e^{-t g}\circ e^{-t f}\circ e^{t g}\circ e^{t f}(q), \quad  \forall q\in M.
$$
\end{definition}
\begin{definition}
An \emph{Hamiltonian} is a smooth function on the cotangent bundle
$$h\in C^\infty(T^*M).
$$
The \emph{Hamiltonian vector field} is the vector field associated to $h$ via the canonical symplectic form $\sigma$
$$\sigma_\lambda (\cdot ,\overrightarrow{h})=d_\lambda h.
$$
Let $(x_1,\ldots,x_n)$ be local coordinates in $M$ and $(\xi_1,\ldots,\xi_n,x_1,\ldots,x_n)$ induced coordinates in $T^*M,\ \lambda=\sum_{i=1}^n\xi_idx_i$. The \emph{symplectic form} has expression $\sigma=\sum^n_{i=1}d\xi_i\wedge dx_i$. Thus, in canonical coordinates, the Hamiltonian vector field has the following form
$$\overrightarrow{h}=\sum^n_{i=1}\left( \frac{\partial h}{\partial \xi_i}\frac{\partial}{\partial x_i} -\frac{\partial h}{\partial x_i}\frac{\partial}{\partial \xi_i}\right).
$$
So the \emph{Hamiltonian system}, which corresponds to $h$, is
$$\dot{\lambda}=\overrightarrow{h}(\lambda), \quad  \lambda \in T^*M,
$$
therefore, in canonical coordinates, it is
$$\left\lbrace
\begin{array}{l}
\dot{x}_i=\frac{\partial h}{\partial \xi_i}\\
\dot{\xi_i}=-\frac{\partial h}{\partial x_i}
\end{array}
\right.
$$
for $i=1,\ldots,n$.
\end{definition}
\begin{definition}
The \emph{Poisson brackets} $\{a,b\}\in \mathcal{C}^\infty(T^*M)$ of two Hamiltonians $a,b\in \mathcal{C}^\infty(T^*M)$ are defined as follows: $\{a,b\}=\sigma(\vec a,\vec b)$; the coordinate expression is:
$$\{a,b\}=\sum_{k=1}^n\left( \frac{\partial a}{\partial \xi_k}\frac{\partial b}{\partial x_k}-\frac{\partial a}{\partial x_k}\frac{\partial b}{\partial \xi_k}\right).
$$
\end{definition}
\begin{remark}
\label{poisson,lie}
Let us recall that, given $g_1$ and $g_2$ vector fields in $M$, considering the Hamiltonians $a_1(\xi,x)=\left\langle \xi, g_1(x)\right\rangle $ and $a_2(\xi,x)=\left\langle \xi, g_2(x)\right\rangle $, it holds
$$\{a_1,a_2\}(\xi,x)=\left\langle \xi, [g_1,g_2](x)\right\rangle.
$$
\end{remark}
\begin{remark}
\label{derivPoiss}
Given a smooth function $\Phi$ in $\mathcal{C}^\infty(T^*M)$, and $\lambda(t)$ solution of the Hamiltonian system $\dot{\lambda}=\overrightarrow{h}(\lambda)$, the derivative of $\Phi(\lambda(t))$ with respect to $t$ is the following
$$\frac{d}{dt}\Phi(\lambda(t))=\{h,\Phi\}(\lambda(t)).
$$
\end{remark}
\subsubsection{Pontryagin Maximum Principle}
Now we give the statement of the Pontryagin Maximum Principle for the time-optimal problem:
\begin{theorem}[Pontryagin Maximum Principle]
Let an admissible control $\tilde{u}$, defined in the interval $t\in [0,\tau_1 ]$, be time-optimal for the system (\ref{control.system}), and let the Hamiltonian associated to this control system be the action on $f_u(q)\in T^*_q M$ of a covector $\lambda\in T^*_q M$: $$\mathcal{H}_u(\lambda)=\left\langle \lambda,f_u(q)\right\rangle . $$
Then there exists $\lambda(t)\in T_{q_{\tilde{u}}(t)}^*M$, for $t\in [0,\tau_1 ]$, never null and lipschitzian, such that for almost all $t\in [0,\tau_1 ]$ the following conditions hold:
\begin{enumerate}
	\item $\dot{\lambda}(t)=\vec{\mathcal{H}}_{\tilde{u}}( \lambda(t))$
	\item $\mathcal{H}_{\tilde{u}}(\lambda(t))= \max_{u\in U} \mathcal{H}_u(\lambda(t))$
\item $\mathcal{H}_{\tilde{u}}(\lambda(t))\geq0$.
\end{enumerate}
Moreover the second condition is called \emph{maximality condition}, and $\lambda(t)$ is called \emph{extremal}.
\end{theorem}
\begin{remark}
Given the canonical projection $\pi:TM\rightarrow M$, we denote $q(t)=\pi(\lambda(t))$ the \emph{extremal trajectory}.
\end{remark}
\subsubsection{Goh condition}
\label{subsecgoh}
Finally, we present the Goh condition, on the singular arcs of the extremal trajectory, in which we do not have information from the maximality condition of the Pontryagin Maxinum Principle. We state the Goh condition only for affine control systems (2.3).
\begin{theorem}[\emph{Goh condition}]
\label{Goh.condition}
Let $\tilde q(t),\ t\in[0,t_1]$ be a time-optimal trajectory corresponding to a control $\tilde u$. If $\tilde u(t)\in\mathrm{int}U$ for any $t\in(\tau_1,\tau_2)$,
then there exist an extremal $\lambda(t)\in T_{q(t)}^*M$ such that
\begin{equation}
\label{cond.di.Goh}
\left\langle \lambda(t),[f_i,f_j](q(t))\right\rangle =0,\quad\ t\in(\tau_1,\tau_2),\ i,j=1,\ldots,m.
\end{equation}
\end{theorem}

\subsection{Consequence of the optimality conditions. The 3-D case with 2-D control. }
In this paper we are going to investigate the local regularity of time-optimal trajectories for the following $3$-dimensional affine control system with a $2$-dimensional control:
\begin{equation}
\label{246}
\dot{q}=f_0(q)+u_1f_1(q)+u_2f_2(q),
\end{equation}
where space of control parameters $U$ is the $2$-dimensional disk.\\
By the Pontryagin Maximum Principle, every time-optimal trajectory of our system has an extremal that is a lift in the cotangent bundle $T^*M$. The extremal satisfies a Hamiltonian system, given by the Hamiltonian defined from the maximality condition.
\begin{notation}
\label{notaz}
Let us call $h_i(\lambda)=\left\langle \lambda,f_i(q) \right\rangle $, $f_{ij}(q)=[f_i,f_j](q),\ f_{ijk}(q)=[f_i,[f_j,f_k]](q)$, $h_{ij}(\lambda)=\left\langle \lambda,f_{ij}(q) \right\rangle $, and $h_{ijk}(\lambda)=\left\langle \lambda,f_{ijk}(q) \right\rangle$, with $\lambda\in T^*_{q}M$ and $i,j,k\in\{0,1,2\}$.
\end{notation}
\begin{definition}
The \emph{singular locus} $\Lambda \subseteq T^*M$, is defined as follows:
$$
\Lambda=\{\lambda\in T^*M : h_1(\lambda)=h_2(\lambda)=0\}.
$$
\end{definition}
The following proposition is an immediate Corollary of the Pontryagin Maximum Principle.

\begin{proposition}
\label{776}
If an extremal $\lambda(t),\ t\in[0,t_1]$, does not intersect the singular locus $\Lambda$, then
\begin{equation}
\label{time-opt.control}
\tilde{u}(t)=\left( \begin{array}{c}
\frac{h_1(\lambda(t))}{(h^2_2(\lambda(t))+h^2_2(\lambda(t)))^{1/2}}\\
 \frac{h_2(\lambda(t))}{(h^2_2(\lambda(t))+h^2_2(\lambda(t)))^{1/2}}
 \end{array}
 \right).
\end{equation}
Moreover, this extremal is a solutions of the Hamiltonian system defined by the Hamiltonian $\mathcal{H}(\lambda)=h_0(\lambda)+\sqrt{h^2_1(\lambda)+h^2_2(\lambda)}$. Thus, it is smooth.
\end{proposition}

\begin{definition}
We will call \emph{bang arc} any smooth arc of a time-optimal trajectory $q(t)$, whose correspondent time-optimal control $\tilde{u}$ lies in the boundary of the space of control parameters, $\tilde{u}(t)\in \partial U$.
\end{definition}
\begin{corollary}
\label{017}
An arc of a time-optimal trajectory, whose extremal is out of the singular locus, is a bang arc.
\end{corollary}
\begin{proof}
From Proposition \ref{776}, given an arc of a time-optimal trajectory $q(t)$, whose extremal $\lambda(t)$ does not intersect the singular locus, its control $\tilde{u}(t)$ satisfies the equation (\ref{time-opt.control}), as a consequence  the arc is smooth with respect to the time. Moreover the time-optimal control belongs to the boundary of $U$. Hence the arc of $q(t)$ that we are considering is a bang arc.
\end{proof}
From Corollary \ref{017} we already have an answer about the regularity of time-optimal trajectories:  every time-optimal trajectory, whose extremal lies out of the singular locus, is smooth.\\
However, we do not know what happen if an extremal touches the singular locus, optimal controls can be not always smooth, hence let us give the following definitions.
\begin{definition}
A \emph{switching} is a discontinuity of an optimal control.\\
Given $u(t)$ an optimal control, $\bar{t}$ is a \emph{switching time} if $u(t)$ is discontinuous at $\bar{t}$.\\ Moreover given $q_u(t)$ the admissible trajectory, $\bar{q}=q_u(\bar{t})$ is a \emph{switching point }if $\bar{t}$ is a switching time for $u(t)$.
\end{definition}

A concatenation of bang arcs is called \emph{bang-bang trajectory}.

An arc of an optimal trajectory that admits an extremal totally contained in the singular locus $\Lambda$, is called \emph{singular arc}.

\section{Statement of the result}
In the rest of the paper, we always assume that $\dim M=3$ and study the time-optimal problem for the system
\begin{equation}
\label{111}
\dot{q}=f_0(q)+u_1f_1(q)+u_2f_2(q),\quad (u_1,u_2)\in U,
\end{equation}
where $f_0$, $f_1$ and $f_2$ are smooth vector fields, $U=\{(u_1,u_2)\in\mathbb R^2: u_1^2+u_2^2\le1\}$; we also assume that $f_1$ and $f_2$ are everywhere linearly independent, and $f_{ij}=[f_i,f_j]$ with $i,j\in\{0,1,2\}$.

\label{Satement}
\begin{theorem}
\label{theorem.results}
Let $\bar{q}\in M$; if
\begin{equation}
\label{5431}
\mathrm{rank}\{f_1(\bar{q}),f_2(\bar{q}),f_{01}(\bar{q}), f_{02}(\bar{q}),f_{12}(\bar{q})\}=3,
\end{equation}
then there exists a neighbourhood $O_{\bar{q}}$ of $\bar{q}$ in $M$ such that 
any time-optimal trajectory contained in $O_{\bar{q}}$ is bang-bang, with not more than one switching.

\end{theorem}

From Corollary \ref{017} we already know that every arc of a time-optimal trajectory, whose extremal lies out of $\Lambda$, is bang, and so smooth. \\
Thus, we are interested to study arcs of a time-optimal trajectories, whose extremals passes through $\Lambda$ or lies in $\Lambda$. \\
We are going to study directly the behaviour of extremals in the cotangent bundle in the neighbourhood of $\bar{\lambda}$, that is any lift of $\bar{q}$ in $\Lambda_{\bar{q}}\subseteq T^*_{\bar{q}}M$, not null.\\
Let us give an equivalent condition to (\ref{5431}) at the point $\bar{\lambda}$.
\begin{claim}
Given $\bar{\lambda}\in \Lambda_{\bar{q}}\subseteq T^*_{\bar{q}}M$, $\bar{\lambda}\neq 0$, the equation (\ref{5431}) is equivalent to
\begin{equation}
\label{54310}
h^2_{01}(\bar{\lambda})+h^2_{02}(\bar{\lambda})+h^2_{12}(\bar{\lambda})\neq 0,
\end{equation}
that does not depend on the choice of $\bar{\lambda}$.
\end{claim}
\begin{proof}
Since by construction $\bar{\lambda}$ is orthogonal to $f_1(\bar{q})$ and $f_2(\bar{q})$, (\ref{5431}) will be true if and only if the valuers $h_{01}(\bar{\lambda})$ $h_{02}(\bar{\lambda})$ and $h_{12}(\bar{\lambda})$ can not be all null.
\end{proof}
In this paper we are going to present exactly in which cases there could appear switchings, with respect to the choice of the triples $(f_0,f_1,f_2)\in V^\infty_{3}(M)$.\\
Let us give the following notation.
\begin{notation}
\label{246} Chosen $\bar{\lambda}\in\Lambda|_{\bar q}$, such that $\bar{\lambda}=f_1(\bar q)\times f_2(\bar q)  $, we introduce the following abbreviated notations: $r:=(h^2_{01}(\bar{\lambda})+h^2_{02}(\bar{\lambda}))^{1/2},\
h_{12}:=h_{12}(\bar{\lambda})$.
\end{notation}
The fist step is to investigate if our system admits singular arcs.
\begin{proposition}
\label{809}
Assuming condition (\ref{54310}), if $r^2\neq h^2_{12}$ there are no optimal extremals in $O_{\bar{\lambda}}$ that lie in the singular locus $\Lambda$ for a time interval. On the other hand, if $r^2= h^2_{12}$ there might be arcs of  optimal extremal contained in $\Lambda$.
\end{proposition}
Thanks to Proposition \ref{809}, if $r^2\neq h^2_{12}$ every optimal extremal could either remain out of the singular locus or intersect it transversally. Consequently, in a neighbourhood of $\bar{\lambda}$ we are allowed to study the solutions of the Hamiltonian system, defined by $\mathcal{H}(\lambda)=h_{0}(\lambda)+\sqrt{h^2_{1}(\lambda)+h^2_{2}(\lambda)}$, that has a discontinuous right-hand side at $\bar{\lambda}$.\\
With this approach we proved the following result.
\begin{theorem}
\label{resultneq}
Assume that condition (\ref{54310}) is satisfied, and suppose that $r^2\neq h^2_{12}$.\\
If
\begin{equation}
\label{54310>}
r^2>h^2_{12},
\end{equation}
then there exist a neighborhood $O_{\bar\lambda}\subset T^*M$ and an interval $(\alpha,\beta),\ \alpha<0<\beta,$ such that for any $z\in O_{\bar\lambda}$ there exists a unique extremal $t\mapsto\lambda(t;z)$ with the initial condition $\lambda(0;z)=z$ defined on the interval $t\in(\alpha+\hat{t},\beta+\hat{t})$, with $\hat{t}\in(-\beta,-\alpha)$. Moreover, $\lambda(t;z)$ continuously depends on $(t,z)\in(\alpha,\beta)\times O_{\lambda}$ and every extremal in $O_{\bar{\lambda}}$ that passes through the singular locus is piece-wise smooth with only one switching. Besides that, we have:
\begin{equation}
\label{jumpo}
u(\bar{t}\pm 0)=\frac 1{r^2}\left( -h_{02}h_{12}\pm h_{01}(r^2-h^2_{12})^{\frac 12},h_{01}h_{12}\pm  h_{02}(r^2-h^2_{12})^{\frac 12}\right),
\end{equation}
where $u$ is the control correspondent to the extremal that passes through $\bar\lambda$, and $\bar t$ is its switching time.
If
\begin{equation}
\label{54310<}
r^2<h^2_{12},
\end{equation}
then there exists a neighborhood $O_{\bar\lambda}\subset T^*M$ such that every optimal extremal does not intersect singular locus in $O_{\bar\lambda}$; all close to $\bar q$ optimal trajectories are smooth bang arcs.
\end{theorem}
\begin{remark}
We would like to stress the fact that formula (\ref{jumpo}) explicitely describes the jump of the time-optimal control at the switching point in terms of Lie brackets relations. \\
If the value $h_{12}$ equals zero at the jump point, then the control reachs the antipodal point of the boundary of the disk. This happen at points where $f_1$ $f_2$ and $f_{12}$ are linearly dependent.\\
Moreover, if the inequality $r^2>h^2_{12}$ is close to be an equality the jump will be smaller and smaller.
\end{remark}
\begin{remark}
 In general, the flow of switching extremals from Theorem~3.6 is not locally Lipschitz with respect to the initial value. A straightforward calculation shows that it is not locally Lipschitz already in the following simple example:
$$
\dot{x}=\left(\begin{array}{c}
0\\0\\\alpha x_1
\end{array}\right)+u_1\left(\begin{array}{c}
1\\0\\0
\end{array}\right)+u_2\left(\begin{array}{c}
0\\1\\ x_1
\end{array}\right)
$$
with $\alpha>1$.
\end{remark}

Since the Pontryagin Maximum Principle is a necessary but not sufficient condition of optimality, even if we have found extremals that passes through the singular locus, we cannot guaranty that they are all optimal, namely that their projections in $M$ are time-optimal trajectory. In some cases they are certainly optimal, in particular, for linear system with an equilibrium target, where to be an extremal is sufficient for optimality. We plan to study general case in a forthcoming paper.\\

In the limit case $r^2=h^2_{12}$ we have the following result:
\begin{proposition}
\label{resulteq}
If
\begin{equation}
\label{54310=}
r^2 = h^2_{12},
\end{equation}
there exists a nieghborhood of $\bar{q}$ such that any time-optimal trajectory that contains $\bar{q}$ and is contained in the neighborhood is a bang arc. The correspondent extremal either remains out of the singular locus $\Lambda$, or lies in
\begin{equation}
\label{020}
\Lambda\cap \{\lambda\,|\, h^2_{01}(\lambda)+h^2_{02}(\lambda)=h^2_{12}(\lambda)\}.
\end{equation}
Anyway, the correspondent optimal control will be smooth without any switching, taking values on the boundary of $U$, in both cases.
\end{proposition}

\begin{remark}
One can notice that the case, in which an extremal $\lambda(t)$ lies in (\ref{020}) for a time interval, is very rare. Indeed, necessarily along the curve it holds conditions $(P_k)$ on $(f_0,f_1,f_2)$, that come from the following equalities
$$\frac{d^k}{dt^k}\left( h^2_{01}(\lambda(t))+h^2_{02}(\lambda(t))-h^2_{12}(\lambda(t)) \right)=0, \quad k\in\mathbb{N},
$$
and it is easy to see that at least conditions $(P_0)$ $(P_1)$ and $(P_2)$ are distinct and independent.
\end{remark}

\section{Proof}
In this Section we are going to present at first the proof of Theorem \ref{resultneq}, secondly we are going to prove Proposition \ref{809}, and finally Proposition \ref{resulteq}. All together, these statements contain Theorem~3.1.
\subsection{Proof of Theorem \ref{resultneq}} Let us present the Blow-up technique, in order to analyse the discontinuous right-hand side Hamiltonian system, defined by
\begin{equation}
\label{111098}
\mathcal{H}(\lambda)=h_0(\lambda) +  \sqrt{h_1^2(\lambda)+h_2^2(\lambda)},
\end{equation}in a neighbourhood $O_{\bar{\lambda}}$ of $\bar{\lambda}$. Secondly, we are going to show the proof of the Theorem if $r^2< h^2_{12}$, and finally we prove it if  $r^2>h^2_{12}$.
\subsubsection{Blow-up technique}In view of the fact that this is a local problem in $O_{\bar{\lambda}}\subseteq T^*M$, it is very natural consider directly its local coordinates $(\xi,x)\in \mathbb{R}^{3*}\times \mathbb{R}^3$, such that $\bar{\lambda}$ corresponds to $(\bar{\xi}, \bar{x})$ with $\bar{x}=0$. Hence,
\begin{equation}
\label{1110988}
\mathcal{H}(\xi,x)=h_0(\xi,x) +  \sqrt{h_1^2(\xi,x)+h_2^2(\xi,x)}.
\end{equation}
Since $f_1$ and $f_2$ are linearly independent everywhere, we can define the never null vector field $f_3$, such that $$f_3(x)=f_1(x)\times f_2(x),$$with the correspondent $h_3(\xi,x)=\left\langle \xi,f_3(x) \right\rangle $. Therefore, we are allowed to consider the following smooth change of variables
$$
\Phi:\, (\xi,x)\longrightarrow  ((h_1,h_2,h_{3}),x),
$$
so the singular locus becomes the subspace $$\Lambda=\{((h_1,h_2,h_{3}),x) \, |\, h_1=h_2=0 \}.$$
\begin{notation}In order not to do notations even more complicated, we call $\lambda$ any point defined with respect to the new coordinates $((h_1,h_2,h_{3}),x)$, and $\bar{\lambda}$ what corresponds to the singular point that we fixed at Notation \ref{246}.
\end{notation}
\begin{definition}
The \emph{blow-up} technique is defined in the following way:\\
We make a change of variables: $(h_1,h_2)=(\rho\cos\theta,\rho\sin\theta)$. Instead of considering the components $h_1$ and $h_2$ of the singular point $\bar{\lambda}$ in $\Lambda$, as the point $(0,0)$ in the 2-dimensional plane, we will consider it as a circle $\{\rho=0\}$, and we denote every point of this circle $\bar{\lambda}_{\theta}$, with respect to the angle.
\end{definition}
\begin{center}
\begin{pspicture}(-0.82,-1.6812989)(6.82,1.6812989)
\rput[bl](0.0,-1.4653614){$\textcolor{black}{\bar{\lambda}}$}
\rput[bl](5.9,-0.07463871){$\textcolor{black}{\theta}$}
\rput[bl](2.4533334,-0.10536131){\large{Blow - up}}
\rput[bl](4.24,-1.6053613){$\textcolor{black}{\bar{\lambda}}$}
\rput[bl](4.4866667,0.37463867){$\textcolor{black}{\bar{\lambda}_{\theta}}$}
\psdots[linecolor=black, dotsize=0.18](0.40666667,-0.45869467)
\psline[linecolor=black, linewidth=0.04, linestyle=dashed, dash=0.17638889cm 0.10583334cm, arrowsize=0.05291667cm 2.0,arrowlength=1.4,arrowinset=0.0]{->}(1.96,-0.45869467)(4.32,-0.45869467)
\pscircle[linecolor=black, linewidth=0.04, dimen=outer](5.653333,-0.4986947){0.39333335}
\pswedge[linecolor=black, linewidth=0.04, dimen=outer](5.6466665,-0.4986947){0.14666666}{0.0}{130.5483}
\psline[linecolor=black, linewidth=0.04](5.66,-0.52536136)(5.4066668,-0.23202801)
\psdots[linecolor=black, dotsize=0.18](5.42,-0.20536135)
\psline[linecolor=black, linewidth=0.04, linestyle=dotted, dotsep=0.10583334cm, arrowsize=0.05291667cm 2.0,arrowlength=1.4,arrowinset=0.0]{->}(0.15333334,-1.0453613)(0.32666665,-0.632028)
\psline[linecolor=black, linewidth=0.04, linestyle=dotted, dotsep=0.10583334cm, arrowsize=0.05291667cm 2.0,arrowlength=1.4,arrowinset=0.0]{->}(4.78,-1.2453613)(5.1,-0.8586947)
\psline[linecolor=black, linewidth=0.04, linestyle=dotted, dotsep=0.10583334cm, arrowsize=0.05291667cm 2.0,arrowlength=1.4,arrowinset=0.0]{->}(4.7,0.35463864)(5.2066665,-0.018694684)
\rput[bl](-1.36,1.0146386){\Large{$\textcolor{black}{((h_1,h_2,h_3), x)}$}}
\psline[linecolor=black, linewidth=0.037, arrowsize=0.05291667cm 2.0,arrowlength=1.4,arrowinset=0.0]{->}(1.96,1.2146386)(4.32,1.2146386)
\rput[bl](4.4866667,1.0146386){\Large{$\textcolor{black}{((\rho,\theta, h_3),x)}$}}
\end{pspicture}

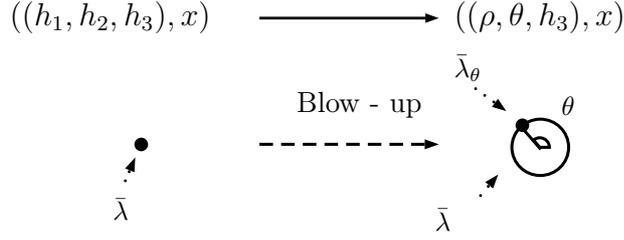
\captionof{figure}{Blow-up technique}
\end{center}
\bigskip

In order to write explicitly the Hamiltonian system of (\ref{111098}) out of $\Lambda$ with this new formulation, let us notice the following aspects.

As it is already know from Proposition \ref{776}, every optimal control $\tilde{u}$ correspondent to an extremal $\lambda(t)$ that lies out of $\Lambda$ satisfies formula (\ref{time-opt.control}), therefore in this new notation it holds $$\tilde{u}(t)=(\cos(\theta(t)),\sin(\theta(t))),$$
where $\theta(t)$ is the $\theta$-component of $\lambda(t)$.\\
Consequently, it is useful denote
$$ f_{\theta}(x)=\cos(\theta) f_1(x)+\sin(\theta) f_2(x) $$
and $h_{\theta}(\lambda)=\left\langle  \xi, f_{\theta}(x) \right\rangle $.\\
Finally we can see that $$h_{\theta}(\lambda)=\sqrt{h^2_1+h^2_2},$$ namely $h_{\theta}(\lambda)=\rho$, because $h_{\theta}(\lambda)=\cos(\theta)h_1+\sin(\theta)h_2$, $\cos(\theta)=\frac{h_1}{\sqrt{h^2_1+h^2_2}}$ and $\sin(\theta)=\frac{h_2}{\sqrt{h^2_1+h^2_2}}$.\\
Hence, with this new formulation the maximised Hamiltonian becomes
\begin{equation}
\label{111099888}
\mathcal{H}(\lambda)=h_0(\lambda)+  h_{\theta}(\lambda),
\end{equation}
and, thanks to Remarks \ref{derivPoiss} and \ref{poisson,lie}, the Hamiltonian system has the following form:
\begin{equation}
\label{11100}
\left\lbrace \begin{array}{l}
\dot{x}=f_0(x)+f_{\theta}(x)\\
\dot{\rho} = h_{0\theta}(\lambda)\\
\dot{\theta} = \frac{1}{\rho}\left( h_{12}(\lambda)+\partial_\theta h_{0\theta}(\lambda) \right)\\
\dot{h}_{3}=h_{03}(\lambda)+h_{\theta 3}(\lambda)
\end{array}\right.
\end{equation}
where $h_{0\theta}(\lambda)=\cos(\theta)h_{01}(\lambda)+\sin(\theta)h_{02}(\lambda)$, and $\partial_\theta h_{0\theta}(\lambda)=\cos(\theta)h_{02}(\lambda)-\sin(\theta)h_{01}(\lambda)$.
\begin{claim}
\label{349}
At the singular point $\bar{\lambda}$ the function $\theta\mapsto h_{12}+\cos(\theta)h_{02}-\sin(\theta)h_{01}$ has two, one or no zeros, if $r^2>h^2_{12}$, $r^2=h^2_{12}$ or $r^2<h^2_{12}$ correspondently.
\end{claim}
\begin{proof}
 We set: $(h_{01}, h_{02})=r(\cos(\phi),\sin(\phi))$; then $h_{12}+\cos(\theta)h_{02}-\sin(\theta)h_{01}=0$ if and only if $\sin(\theta-\phi)=\frac{h_{12}}{r}$.
\end{proof}
\begin{lemma}
\label{lemma1}
Given the singular point $\bar{\lambda}$ and $\bar{\lambda}_{\theta_i}$ such that it holds $h_{12}+\cos(\theta_i)h_{02}-\sin(\theta_i)h_{01}=0$. We consider $O_{\bar{\lambda}_{\theta_i}}$ neighborhoods of $\bar{\lambda}_{\theta_i}$, and a neighbourhood $O_{\bar{\lambda}}$ small enough such that $\forall \hat{\lambda}_{\hat{\theta}}\in O_{\bar{\lambda}}\setminus \overline{\cup_{i}O_{\bar{\lambda}_{\theta_i}}}$ it holds $h_{12}(\hat{\lambda})+\cos(\hat{\theta})h_{02}(\hat{\lambda})-\sin(\hat{\theta})h_{01}(\hat{\lambda})\neq 0$. For each connected component $O$ of $O_{\bar{\lambda}}\setminus \overline{\cup_{i}O_{\bar{\lambda}_{\theta_i}}}$ there exist constants  $c>0$ and $\alpha>0$ such that if an extremal $\lambda(t)$ lies in $O$ for a time interval $I=(0,\tau_1)$, with $\lambda(0)\not \in \Lambda$, then it  holds the following inequality: $\rho(t)\geq c e^{-\alpha t}\rho(0)$, for $t\in I$.
\end{lemma}
\begin{proof}
Without loss of generality let us study a connected component $O$ of $O_{\bar{\lambda}}\setminus \overline{\cup_{i}O_{\bar{\lambda}_{\theta_i}}}$ where $$h_{12}(\lambda)+\partial_\theta h_{0\theta}(\lambda) > 0.$$
Since in $O$ the map $\lambda \rightarrow h_{12}(\lambda)+\partial_\theta h_{0\theta}(\lambda)$ is continuous and not null, it is bounded, then there exist constants $c_1>0$ and $c_2>0$ such that
$$c_1 \geq h_{12}(\lambda)+\partial_\theta h_{0\theta}(\lambda)\geq c_2>0.
$$
Given the extremal $\lambda(t)$ in $O$, we can observe that
$$\frac{d}{dt}\left[\rho(t)\left[h_{12}(\lambda(t)) + \partial_\theta h_{0\theta}(\lambda(t))\right]\right] =\rho(t)A(\lambda(t))
$$
where
$$A(\lambda(t))=\dot{h}_{12}(\lambda(t))+\cos(\theta(t))\dot{h}_{02}(\lambda(t))-\sin(\theta(t))\dot{h}_{01}(\lambda(t)).$$
Moreover, we can claim that $A_{|O}$ is bounded from below by a negative constant $C$
$$A_{|O}\geq C,
$$
due to the facts that, by Remark \ref{derivPoiss} $$\dot{h}_{ij}(\lambda(t))=h_{0ij}(\lambda(t))+\cos(\theta(t))h_{1ij}(\lambda(t))+\sin(\theta(t))h_{2ij}(\lambda(t)),$$ and any function $h_{kij}(\lambda)$ is continuous in $\bar{\lambda}$, for each indexes $i,j,k \in \{0,1,2\}$.\\
Finally, we can see that
$$\frac{d}{dt}\left[\frac{\rho(t)\left[h_{12}(\lambda(t))+ \partial_\theta h_{0\theta}(\lambda(t))\right]}{\mathrm{exp}\left( \int^t_0 C \left[h_{12}(\lambda(t))+ \partial_\theta h_{0\theta}(\lambda(t))\right]^{-1}ds \right)}\right]\geq 0,
$$
hence, for each $t\geq 0$, by the monotonicity:
$$
\begin{array}{rcl}
\rho(t)&\geq& \rho(0)\,\,\frac{h_{12}(\lambda(0))+ \partial_\theta h_{0\theta}(\lambda(0))}{h_{12}(\lambda(t))+ \partial_\theta h_{0\theta}(\lambda(t))}\,\,\mathrm{exp}\left( \int^t_0 C \left[h_{12}(\lambda(s))+ \partial_\theta h_{0\theta}(\lambda(s))\right]^{-1}ds \right)\\
&\geq&\rho(0)\,\,\frac{c_2}{c_1}\,\,\mathrm{exp}\left( \frac{C}{c_2}t \right).
\end{array}
$$
Denoting $c:=\frac{c_2}{c_1}$ and $\alpha:=-\frac{C}{c_2}$, the thesis follows.
\end{proof}
\subsubsection{The $r^2< h^2_{12}$ case}\textcolor{white}{.}\\
Lemma \ref{lemma1} and Claim \ref{349} immediately imply the following Corollary:
\begin{corollary}
If we assume conditions (\ref{54310}) and $r^2< h^2_{12}$, given $O_{\bar{\lambda}}$ small enough there exist two constants $c>0$ and $\alpha>0$ such that every extremal that lies for a time interval $I$ in $O_{\bar{\lambda}}$ satisfies the following inequality: $\rho(t)\geq c e^{-\alpha t}\rho(0)$, for $t\in I$.
\end{corollary}
This Corollary proves the $r^2<h^2_{12}$ case of the Theorem, because it shows that, given this condition, every optimal extremal in $O_{\bar{\lambda}}$ does not intersect the singular locus  in finite time, and forms a smooth local flow.
\subsubsection{The $r^2>h^2_{12}$ case}
\begin{proposition}
Assuming conditions (\ref{54310}) and (\ref{54310>}), there exists a unique extremal that passes through $\bar{\lambda}$ in finite time.
\end{proposition}
\begin{proof}
Let us prove that there is a unique solution of the system (\ref{11100}) passing through its point of discontinuity $\bar{\lambda}$ in finite time.\\
In order to detect solutions that go through $\bar{\lambda}$, we rescale the time considering the time $t(s)$ such that $\frac{d}{ds}t(s)=\rho(s)$ and we obtain the following system
\begin{equation}
\label{11100s}
\left\lbrace \begin{array}{l}
x'=\rho\left(f_0(x)+f_{\theta}(x)\right)\\
\rho' = \rho h_{0\theta}\\
\theta' =h_{12}+\frac{\partial}{\partial \theta}h_{0\theta} \\
h'_{3}=\rho\left(h_{03}+h_{\theta3}\right),
\end{array}\right.
\end{equation}
with a smooth right-hand side.\\
This system has an invariant subset $\{\rho=0\}$ in which only the $\theta$-component is moving. Moreover, as we saw from Claim \ref{349}, at $\bar{\lambda}$ there are two equilibria $\bar{\lambda}_{\theta_-}=((0,\theta_-,1),\bar{x})$ and $\bar{\lambda}_{\theta_+}=((0,\theta_+,1),\bar{x})$, such that $\sin(\theta_{\pm}-\phi)=\frac{h_{12}}{r}$ and $\cos(\theta_{\pm}-\phi)=\pm\frac{\sqrt{r^2-h^2_{12}}}{r}$.

\bigskip
Let us present the Shoshitaishvili's Theorem \cite{sh} that explain how is the behaviour of the solutions in $\bar{\lambda}_{\theta_-}$ and $\bar{\lambda}_{\theta_+}$.
\begin{theorem}[Shoshitaishvili's Theorem]
In $\mathbb{R}^n$, let the $C^k$-germ, $2\leq k<\infty$, of the family
\begin{equation}
\label{222}
\left\lbrace \begin{array}{l}
\dot{z}=Bz+r(z,\varepsilon),\\
\dot{\varepsilon}=0, z\in\mathbb{R}^n,\varepsilon\in\mathbb{R}^l,
\end{array}\right.
\end{equation}
be given, where $r\in C^k(\mathbb{R}^n\times \mathbb{R}^l)$, $r(0,0)=0$, $\partial_{z}r_{|(0,0)}=0$, and $B:\mathbb{R}^n\rightarrow \mathbb{R}^n$ is a linear operator whose eigenvalues are divided into three groups:
$$
\begin{array}{c}
\begin{array}{rcl}
\mathrm{I}&=&\{ \lambda_i, 1\leq i \leq k^0|\, \mathrm{Re}\lambda_i=0 \}\\
\mathrm{II}&=&\{ \lambda_i, k^0+1\leq i \leq k^0+k^-|\,\mathrm{Re}\lambda_i<0 \}\\
\mathrm{III}&=&\{ \lambda_i, k^0+k^-+1\leq i \leq k^0+k^-+k^+|\,\mathrm{Re}\lambda_i>0 \}
\end{array}
\\
\\
\\
k^0+k^-+k^+=n.
\end{array}
$$
Let the subspaces of $\mathbb{R}^n$, which are invariant with respect $B$ and which correspond to these groups be denoted by $X$, $Y^-$ and $Y^+$ respectively, and let $Y^-\times Y^+$ be denoted by $Y$.\\
Then the following assertions are true:\\
\begin{enumerate}
\item There exists a $C^{k-1}$ manifold $\gamma^0$ that is invariant with respect to the germ (\ref{222}), may be given by the graph of mapping $\gamma^0:X\times \mathbb{R}^l\rightarrow Y$, $y=\gamma^0(x,\varepsilon)$, and satisfies $\gamma^0(0,0)=0$ and $\partial_x \gamma^0(0,0)=0$.
\item The germ of the family (\ref{222}) is homeomorphic to the product of the multidimensional saddle $\dot{y}^+=y^+$, $\dot{y}^-=-y^-$, and the germ of the family
$$\left\lbrace \begin{array}{l}
\dot{x}=Bx+r_1(x,\varepsilon),\\
\dot{\varepsilon}=0,
\end{array} \right.
$$
where $r_1(x,\varepsilon)$ is the $x$-component of the vector $r(z,\varepsilon)$, $z=(x,\gamma^0(x,\varepsilon))$, i.e. the germ of (\ref{222}) is homeomorphic to the germ of the family
$$\left\lbrace \begin{array}{l}
\begin{array}{ll}
\dot{y}^+=y^+, & \dot{y}^-=-y^-
\end{array}\\
\begin{array}{ll}
\dot{x}=Bx+r_1(x,\varepsilon), & \dot{\varepsilon}=0.
\end{array}
\end{array} \right.
$$
\end{enumerate}
\end{theorem}
Let us investigate which are the eigenvalues of the Jacobian of the system (\ref{11100s}).\\
Since $\bar{\lambda}_{\theta_-}$ and $\bar{\lambda}_{\theta_+}$ belong to the invariant subset $\{\rho=0\}$, where the components $\rho$, $h_3$ and $x$ are fixed, at $\bar{\lambda}_{\theta_\pm}$ the eigenvalues are ${\partial_{\rho}\rho'}_{|\bar{\lambda}_{\theta_\pm}}={h_{0\theta}}_{|\bar{\lambda}_{\theta_\pm}}$ and ${\partial_{\theta}\theta'}_{|\bar{\lambda}_{\theta_\pm}}=-{h_{0\theta}}_{|\bar{\lambda}_{\theta_\pm}}$ and four $0$.\\
Moreover, ${h_{0\theta}}_{|\bar{\lambda}_{\theta_-}}$ and ${h_{0\theta}}_{|\bar{\lambda}_{\theta_+}}$ are not null with opposite sign, because it holds
$$\begin{array}{rcl}
{h_{0\theta}}_{|\bar{\lambda}_{\theta_\pm}}&=& \cos(\theta_\pm)h_{01}+\sin(\theta_\pm)h_{02}\\
&=&r \cos(\theta_\pm- \phi)=\pm \sqrt{r^2-h^2_{12}},
\end{array}
$$
namely ${h_{0\theta}}_{|\bar{\lambda}_{\theta_-}}=-\sqrt{r^2-h^2_{12}}$ and ${h_{0\theta}}_{|\bar{\lambda}_{\theta_+}}=\sqrt{r^2-h^2_{12}}$.

Central manifolds $\gamma^0$ of Theorem~4.5 applied to the equilibria $\bar\lambda_{\theta_\pm}$ are 4-dimensional submanifolds defined by the equations $\rho=0,\ \theta=\theta_\pm$. The dynamics on the central manifold is trivial: all points are equilibria. Hence, according to the Shoshitaishvili Theorem, only trajectories from the one-dimensional asymptotically stable (unstable) invariant submanifold tend to the equilibrium point
$\bar\lambda_{\theta_\pm}$ as $t\to+\infty\ (t\to\-\infty)$. Moreover, $\rho=0$ is a 5-dimensional invariant submanifold with a very simple dynamics: $\theta$ is moving from $\theta_-$ to $\theta_+$.

\begin{center}
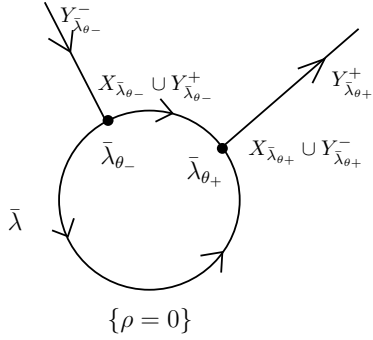

\psscalebox{.8 .8} 
{\begin{pspicture}(-1,-3)(5,3)
\pscircle[linecolor=black, linewidth=0.03, dimen=outer](2.36,-0.6065674){1.5}
\psline[linecolor=black, linewidth=0.03](3.58,0.29343262)(5.8,2.2334325)
\rput[bl](2.94,-0.38656738){\Large{$\bar{\lambda}_{\theta_+}$}}
\rput[bl](1.52,-0.13343262){\Large{$\bar{\lambda}_{\theta_-}$}}
\psline[linecolor=black, linewidth=0.03](1.64,0.7134326)(0.64,2.6734326)
\psline[linecolor=black, linewidth=0.03](0.7,2.0734327)(1.06,1.8534327)(1.1,2.2134326)
\psline[linecolor=black, linewidth=0.03](2.54,1.0734326)(2.74,0.8334326)(2.46,0.6934326)
\psline[linecolor=black, linewidth=0.03](4.8,1.5734326)(5.24,1.7334327)(5.04,1.3334327)
\psline[linecolor=black, linewidth=0.03](3.2,-1.5265674)(3.56,-1.4665674)(3.54,-1.8065674)
\psline[linecolor=black, linewidth=0.03](0.76,-1.1065674)(1.02,-1.2065674)(1.12,-0.9065674)
\rput[bl](5.32,1.0734326){$Y^+_{\bar{\lambda}_{\theta+}}$}
\rput[bl](3.96,-0.06656738){$X_{\bar{\lambda}_{\theta+}}\cup Y^-_{\bar{\lambda}_{\theta+}}$}
\rput[bl](0.86,2.1334327){$Y^-_{\bar{\lambda}_{\theta-}}$}
\rput[bl](1.48,1.0134326){$X_{\bar{\lambda}_{\theta-}}\cup Y^+_{\bar{\lambda}_{\theta-}}$}
\psdots[linecolor=black, dotsize=0.18](1.68,0.7134326)
\psdots[linecolor=black, dotsize=0.18](3.56,0.25343263)
\rput[bl](1.64,-2.8065674){\Large{$\{\rho=0\}$}}
\rput[bl](0.0,-1.1065674){\Large{$\bar{\lambda}$}}
\end{pspicture}}
\captionof{figure}{Description of stable and unstable components of the equilibria. }
\end{center}
\bigskip

Keeping in mind that only a part of the phase portrait where $\rho\ge 0$ is relevant for our study we obtain that exactly one extremal enters submanifold $\rho=0$ at $\bar\lambda_{\theta_-}$ and exactly one extremal goes out of this submanifold at $\bar\lambda_{\theta_+}$. Moreover, the same result in the same neighborhood is valid for any $\hat\lambda\in\Lambda$ sufficiently close to $\bar\lambda$ with $\hat\lambda$ playing the role of parameter $\varepsilon$ in the Shoshitaishvili Theorem.

Finally, we are going to show that the extremal that we found passes through $\bar{\lambda}$ in finite time. At the moment it is known that there exists $\lambda(t(s))$ that satisfies (\ref{11100s}) and it reaches $\bar{\lambda}$ at a equilibrium, so $\lambda(t(s))$ attains and escapes from $\bar{\lambda}$ in infinite time $s$. \\
Thus, let us estimate the time $\Delta t$ that this extremal needs to reach $\bar{\lambda}$.

Due to the facts that ${h_{0\theta}}_{|\bar{\lambda}_{\theta -}}<0$ and $h_{0\theta}$ is continuous in $\bar{\lambda}_{\theta_-}$, there exist a neighbourhood $O_{\bar{\lambda}_{\theta-}}$ of $\bar{\lambda}_{\theta_-}$, in which $h_{0\theta}$ is bounded from above by a negative constant $c_1<0$, namely ${h_{0\theta}}_{|O_{\bar{\lambda}_{\theta-}}}<c_1<0$. \\
Hence, in $O_{\bar{\lambda}_{\theta-}}$ we have the following estimate of the derivative $\rho'$ $$ \rho'=\rho\, h_{0\theta}<\rho \,c_1,$$
consequently until $\rho(s)>0$, it holds $$ \int^s_{s_0} \frac{\rho'}{\rho}ds<\int^s_{s_0} c_1 ds,$$
then this inequality implies $\log(\rho(s))<c_1 (s-s_0)+\log(\rho(s_0))$, and so $$\rho(s)<\rho(s_0)e^{c_1 (s-s_0)}.$$
Since $\frac{d}{ds}t(s)=\rho(s)$, the amount of time that we want to estimate is the following
$$\Delta t=\lim_{s\rightarrow \infty}t(s)-t(s_0)=\int^{\infty}_{s_0} \rho(s) ds,
$$
therefore,
$$\Delta t=\int^{\infty}_{s_0} \rho(s) ds<\rho(s_0)\int^{\infty}_{s_0} e^{c_1 (s-s_0)}ds=\frac{\rho(s_0)}{-c_1}<\infty.
$$
The amount of time in which this extremal goes out from $\bar{\lambda}$ may be estimate in an analogous way.
\end{proof}

By the previous Proposition and the fact that every extremal out of $\Lambda$ is smooth, it is proven that there exist a neighbourhood $O_{\bar\lambda}\subset T^*M$ and an interval $(\alpha,\beta),\ \alpha<0<\beta,$ such that for any $z\in O_{\bar\lambda}$ there exists a unique extremal $t\mapsto\lambda(t;z)$ with the initial condition $\lambda(0;z)=z$ defined on the interval $t\in(\alpha+\hat{t},\beta+\hat{t})$, with $\hat{t}\in(-\beta,-\alpha)$. Furthermore, every extremal in $O_{\bar{\lambda}}$ that passes through the singular locus is piece-wise smooth with only one switching. The control $u$ correspondent to the extremal that passes through $\bar\lambda$ jumps at the switching time $\bar t$ from $u(\bar{t}-0)=(\cos(\theta_-), \sin(\theta_-))$ to $u(\bar{t}+0)=(\cos(\theta_+), \sin(\theta_+))$, hence
$$\begin{array}{rcl}
u(\bar{t}\pm 0)&=&\left( \cos(\phi-(\theta_\pm- \phi)),\sin(\phi-(\theta_\pm -\phi))\right)\\
&=&\left( -\frac{h_{12}}{r} \sin(\phi ) \pm \frac{\sqrt{r^2-h^2_{12}}}{r} \cos(\phi ),\frac{h_{12}}{r} \cos(\phi ) \pm  \frac{\sqrt{r^2-h^2_{12}}}{r} \sin(\phi ) \right).
\end{array}
$$
\bigskip
\\
Let us conclude the proof with the following Proposition.
\begin{proposition}
\label{prop}
The map $(t;z)\rightarrow\lambda(t;z)$, $(t,z)\in(\alpha,\beta)\times O_{\lambda}$ is continuous.
\end{proposition}
\begin{proof}
It remains to prove the continuity  of the flow $(t;z)\rightarrow \lambda(t;z)$ with respect to $z$, thus we prove that for each $\varepsilon>0$ there exists a neighbourhood  $O^\varepsilon_{\bar{\lambda}}$ such that the maximum time interval of the extremals in this neighbourhood $\Delta_{O^\varepsilon_{\bar{\lambda}}}t$ is less than $\varepsilon$.\\
As we saw previously, the extremal through $\bar{\lambda}$ will arrive and go out with angles $\theta_-$ and $\theta_+$, then we can distinguish three parts of the extremals close to $\bar{\lambda}$: the parts in $O_{\bar{\lambda}_{\theta_-}}$ and in $O_{\bar{\lambda}_{\theta_+}}$, and that part in the middle that is close to $\rho=0$.\\
In this last region,  since for each extremal $\rho$ is close to $0$ and the correspondent time interval with time $s$ is bounded, then $\Delta t$ is arbitrarily small with respect to $O_{\bar{\lambda}}$.\\
Hence, in $O_{\bar{\lambda}_{\theta_\pm}}$ we are going to show that there exists a sequence of neighbourhoods of $\bar{\lambda}_{\theta_\pm}$
$$\left(O^R_{\theta_\pm}\right)_R,$$
such that $$\lim_{R\rightarrow 0^+}\Delta_{O^R_{\theta_\pm}}t=0.$$
For simplicity, we are going to prove this fact in $O_{\bar{\lambda}_{\theta_-}}$, because the situations in $O_{\bar{\lambda}_{\theta_-}}$ and $O_{\bar{\lambda}_{\theta_+}}$ are equivalent.\\
Let us denote $O^R_{\theta_-}$ a neighbourhood of $\bar{\lambda}_{\theta_-}$ such that $O^R_{\theta_-}\subseteq O_{\bar{\lambda}_{\theta_-}}$, for each $((\rho,\theta,h_3),x)\in O^R_{\theta_-}$ $\rho<R$ and $|\theta-\theta_-|<R$. Therefore, we can define $$M_R=\sup_{\lambda\in O^R_{\theta_-}}h_{0\theta}(\lambda),$$ and assume that it is strictly negative and finite, due to the fact that  we can choose $O_{\bar{\lambda}_{\theta_-}}$ in which $h_{0\theta}(\lambda)$ is strictly negative and finite.\\
Hence, for every extremal $\lambda(s)$ in $O^R_{\theta_-}$, until its $\rho$-component is different that zero, it holds
$$\frac{\dot{\rho}(s)}{\rho(s)}<M_R,
$$
then $$\rho(s)<\rho(s_0)e^{M_R(s-s_0)},$$ for every $s>s_0$.\\
Consequently, $\Delta_{O^R_{\theta_-}}t$ can be estimated in the following way:
$$\Delta_{O^R_{\theta_-}}t<\int^{\infty}_{s_0}\rho(s_0)e^{M_R(s-s_0)}ds=\frac{\rho(s_0)}{-M_R}<\frac{R}{-M_R}.
$$
Due to the fact that $\lim_{R\rightarrow 0^+}\frac{R}{-M_R}=0$, we have proved that for each $\varepsilon>0$ there exists $O^R_{\theta_-}$ such that $\Delta_{O^R_{\theta_-}}t<\varepsilon$
\end{proof}
\subsection{Proof of Proposition \ref{809} }
Let us assume that there exist a time-optimal control $\tilde{u}$, and an interval $(\tau_1,\tau_2)$ such that $\tilde{u}$ corresponds to an extremal $\lambda(t)$ in $O_{\bar{\lambda}}$, and $\lambda(t)\in \Lambda$, $\forall t\in(\tau_1,\tau_2)$. By construction, for $t\in(\tau_1,\tau_2)$ it holds
\begin{equation}
\label{556}
\left\lbrace
\begin{array}{l}
\frac{d}{dt}h_1(\lambda(t))=0\\
\frac{d}{dt}h_2(\lambda(t))=0.
\end{array}
\right.
\end{equation}
Since the maximized Hamiltonian associated to $\tilde{u}$ is
$$\mathcal{H}_{\tilde{u}}(\lambda)=h_0(\lambda)+\tilde{u}_1h_1(\lambda)+\tilde{u}_2h_2(\lambda),
$$
by Remark \ref{derivPoiss}, (\ref{556}) implies
\begin{equation}
\label{777}\left\lbrace
\begin{array}{l}
h_{01}(\lambda(t))-\tilde{u}_2 h_{12}(\lambda(t))=0\\
h_{02}(\lambda(t))+\tilde{u}_1 h_{12}(\lambda(t))=0.
\end{array}
\right.
\end{equation}
Moreover, due to condition (\ref{54310}) we can claim that $h_{12}(\lambda(t))\neq 0$ along this singular arc, therefore we have an explicit formulation of $\tilde{u}$ in a singular arc
\begin{equation}
\label{607}
\left\lbrace
\begin{array}{l}
\tilde{u}_1=-\frac{h_{02}(\lambda(t))}{h_{12}(\lambda(t))}\\
\tilde{u}_2=\frac{h_{01}(\lambda(t))}{h_{12}(\lambda(t))}.
\end{array}
\right.
\end{equation}
In particular, its norm is the following
$$||\tilde{u}||^2=\frac{h^2_{02}(\lambda(t))+h^2_{01}(\lambda(t))}{h^2_{12}(\lambda(t))}.
$$
If $r^2> h^2_{12}$, we arrive to a contradiction, because in this case $||\tilde{u}||^2>1$ but the norm of admissible controls is less equal than $1$. On the other hand, if $r^2< h^2_{12}$, such extremals might exist, but they are not optimal by the Goh Condition, presented at Subsection \ref{subsecgoh}.\\
Hence, we have proved that if $r^2\neq h^2_{12}$ there are no optimal extremals that lie in $\Lambda$ for a time interval.\\
On the other hand, by these observations, if $r^2= h^2_{12}$, optimal singular arcs could exists.
\subsection{Proof of Proposition \ref{resulteq}}
In the limit case $r^2=h^2_{12}$, by what we have just seen at the proof of Proposition \ref{809}, we can claim that there could be optimal trajectories, whose extremals lie in (\ref{020}), and the correspondent controls take values on the boundary of the disk $U$, with equation (\ref{607}).\\
Now, we are going to show that, given a time-optimal trajectory through $\bar{q}$, whose extremal has a point out of the singular locus, then it does not attain $\Lambda$ in finite time. \\
From Claim \ref{349} and Lemma \ref{lemma1} we already have an estimate of the behaviour of the extremal out of a small neighbourhood of $\bar{\lambda}_{\bar{\theta}}$, where $\bar{\theta}$ is the unique angle such that $h_{12}+\cos(\bar{\theta})h_{02}-\sin(\bar{\theta})h_{01}=0$. 
We are going to extend the estimate to a neighborhood $O_{\bar{\lambda}_{\bar{\theta}}}$ of
$\bar{\lambda}_{\bar{\theta}}$.\\
Without loss of generality, we assume that $\bar{\theta}=0$.\\
Let us omit some boring routine details and focus on the essential part of the estimate. First we freeze slow coordinates $x,h_3$ and study the system (\ref{11100}) with only two variables $\rho,\theta$. In the worst scenario we get the following system:

\begin{equation*}
\label{mel44}\left\lbrace \begin{array}{l}
\dot{\rho} =-\sin(\theta)-\rho\\
\dot{\theta} = \frac{1}{\rho}\left( 1-\cos(\theta) \right)+1.\\
\end{array}\right.
\end{equation*}
Consequently, the behaviour of $\rho$-component with respect the $\theta$-component is described by the following equation:
\begin{equation}
\label{mel45}
\rho'(\theta) =\frac{-\rho(\sin(\theta)+\rho)}{1-\cos(\theta) +\rho}.
\end{equation}
With the next Lemma \ref{lemminoino} we analyse (\ref{mel45}) and prove that, there exist a containing 0 interval $I$ on the $\theta$-axis, on which $\rho$ has a positive increment for any sufficiently small initial condition $\rho(0)=\rho_0>0$.\\
Lemma \ref{lemminoino} with Lemma \ref{lemma1} implies the thesis of Proposition \ref{resulteq}.
\begin{lemma}
\label{lemminoino}
Given $O_{\bar{\lambda}_{\bar{\theta}}}$, there exist $\eta>0$ small enough and $\theta_1>0$, such that for every initial values $(\rho(0),\theta(0))=(\rho_0,0)$ with $\rho_0\neq 0$, the solution of system (\ref{mel45}) satisfies the following implication: if $\theta>\theta_1$ then $$\rho(-\theta)<\rho(\eta\, \theta).$$
\end{lemma}
\begin{proof}
Given any $\eta>0$ and any solution of (\ref{mel45}) $\rho(\theta)$, we are going to compare the behaviour of $\tilde{\rho}(\theta)=\rho(-\theta)$ and $\hat{\rho}(\theta)=\rho(\eta\,\theta)$ for $\theta>0$.\\
They will be solutions for $\theta>0$ of the following two systems
$$
\tilde{\rho}'(\theta)=\frac{\tilde{\rho}(\tilde{\rho}-\sin(\theta))}{1-\cos(\theta)+\tilde{\rho}}
$$
and
$$
\hat{\rho}'(\theta)=-\eta\frac{\hat{\rho}(\hat{\rho}+\sin(\eta\,\theta))}{1-\cos(\eta\,\theta)+\hat{\rho}}.
$$
We can see that $\tilde{\rho}'(0)>\hat{\rho}'(0)$, thus if $\theta$ is very small it holds $\tilde{\rho}(\theta)>\hat{\rho}(\theta)$.\\
On the other hand, let us notice that choosing $\eta>0$ small there exists $\nu>1$ such that if $\theta>\nu\,\rho$ then $\hat{\rho}'(\theta)>\tilde{\rho}'(\theta)$. By the classical theory of dynamical system, this implies that in the domain $$\{(\rho,\theta)\,|\,\theta>\nu \rho\}$$ if $\hat{\rho}(\theta)>\tilde{\rho}(\theta)$ at a certain $\theta>0$, then the inequality remains true for every bigger value.\\
In order to compare the behaviour of $\tilde{\rho}(\theta)$ and $\hat{\rho}(\theta)$ when $\rho_0$ tends to zero, we consider the following re-scaling:
$$\left\lbrace
\begin{array}{l}
\theta=st\\
 \tilde{\rho}=s+s^2x(t)\\
 \hat{\rho}=s+s^2y(t)
\end{array}
\right.
$$
where $s$ is the initial value $\rho_0$ and $x(0)=y(0)=0$.\\
One can easily notice that if $s$ tends to $0$ then
$$\left\lbrace
\begin{array}{l}
x'(t)=1-t+O(s)\\
y'(t)=\eta(-1-\eta t)+O(s),
\end{array}
\right.
$$
hence, it holds
$$\left\lbrace
\begin{array}{l}
x_0(t)=t-\frac{1}{2}t^2+O(s)\\
y_0(t)=-\eta t-\frac{\eta^2}{2}t^2+O(s),
\end{array}
\right.
$$
and
$$x_0(t)-y_0(t)=t\left( (1+\eta) -\frac{(1-\eta^2)}{2}t\right)+O(s).
$$
Hence, there exist $T>2\frac{1+\eta}{1-\eta^2}>2$, such that, denoting $\rho^{\mathrm{MAX}}_0$ the maximum among the initial values $\rho_0$ in $O_{\bar{\lambda}_{\bar{\theta}}}$, and calling $\theta_1=\rho^{\mathrm{MAX}}_0 T$, it holds that if $\theta>\theta_1$ then $\tilde{\rho}(\theta)<\hat{\rho}(\theta)$, namely
$$\rho(-\theta)<\rho(\eta\, \theta).$$
\end{proof}

\end{document}